\documentclass[a4paper, reqno, 11pt, oneside]{amsart}

\usepackage{iftex}

\ifPDFTeX
  \usepackage[T1]{fontenc}
  \usepackage{lmodern}
  \usepackage{xcolor}
\else
  \usepackage[T1]{fontenc}
  \usepackage{lmodern}
  \usepackage{graphicx}
  \usepackage{xcolor}
\fi

\usepackage{tikz}
\usetikzlibrary{positioning,arrows.meta,calc}

\usepackage{amssymb}
\usepackage{amstext}
\usepackage{amsmath}
\usepackage{amscd}
\usepackage{amsthm}
\usepackage{amsfonts}
\usepackage{latexsym}
\usepackage{mathrsfs}
\usepackage{nicematrix}

\usepackage[
  includehead,
  includefoot,
  top=25truemm,
  bottom=25truemm,
  left=25truemm,
  right=25truemm
]{geometry}

\definecolor{darkblue}{rgb}{0.0, 0.0, 0.55}
\definecolor{darkmagenta}{rgb}{0.55, 0.0, 0.55}
\definecolor{darkgreen}{rgb}{0.0, 0.2, 0.13}

\ifPDFTeX
  \usepackage[colorlinks=true]{hyperref}
\else
  \usepackage[dvipdfmx,colorlinks=true]{hyperref}
\fi

\hypersetup{
  linkcolor  = darkblue,
  citecolor  = darkmagenta,
  urlcolor   = darkgreen,
  colorlinks = true,
}

\makeatletter
  
  \@addtoreset{equation}{section}
\makeatother

\makeatletter
\@namedef{subjclassname@2020}{%
  \textup{2020} Mathematics Subject Classification}
\makeatother

\newcommand{\ZZ}{\mathbb{Z}}

\newcommand{\kk}{\Bbbk}

\def\Hilb{\operatorname{Hilb}}
\newcommand{\notmid}{\mathrel{\ooalign{$\mkern-5mu\not$\crcr$|$}}}

\newtheorem{thm}{Theorem}[section]

\newtheorem{lem}[thm]{Lemma}
\newtheorem{prop}[thm]{Proposition}

\theoremstyle{definition}
\newtheorem{dfn}[thm]{Definition}
\newtheorem{ex}[thm]{Example}
\newtheorem{q}[thm]{Question}
\newtheorem{rem}[thm]{Remark}

\begin{document}

\title{On the $h$-polynomials of cyclotomic standard graded commutative algebras}

\author{Akihiro Higashitani}
\address{Department of Pure and Applied Mathematics, 
Graduate School of Information Science and Technology, 
Osaka University, 
1-5, Yamadaoka, Suita, Osaka 565-0871, Japan}
\email{higashitani@ist.osaka-u.ac.jp}

\author{Kenta Ueyama}
\address{Department of Mathematics,
Faculty of Science,
Shinshu University,
3-1-1 Asahi, Matsumoto, Nagano 390-8621, Japan}
\email{ueyama@shinshu-u.ac.jp}


\begin{abstract}
We call a standard graded commutative $\kk$-algebra \textit{cyclotomic} if its $h$-polynomial has all its roots on the unit circle in the complex plane. 
Complete intersections provide typical examples of cyclotomic algebras, 
since the $h$-polynomial of any standard graded complete intersection is a product of polynomials of the form $1 + t + \cdots + t^{m-1}$. 
We refer to such polynomials as being \textit{of type CI}. 
A natural question is whether there exists a cyclotomic standard graded $\kk$-algebra whose $h$-polynomial is not of type CI.
In this paper, we give a partial answer to this question. 
We show that the $h$-polynomial $h_R(t)$ of a cyclotomic standard graded $\kk$-algebra $R$ is of type CI whenever $h_R(1) \in \{1, 4, 6\}$ or $h_R(1)$ is prime. 
On the other hand, if $n$ is a composite integer with $n \geq 8$, then there exists a cyclotomic standard graded $\kk$-algebra $R$ whose $h$-polynomial $h_R(t)$ is not of type CI and satisfies $h_R(1) = n$.
\end{abstract}

\maketitle

\section{Introduction}
Throughout this paper, let $\kk$ denote a field. We say that $R$ is a \emph{standard graded} $\kk$-algebra 
if $R = \bigoplus_{n \ge 0} R_n$ is a commutative graded $\kk$-algebra with $R_0 = \kk$ and finitely generated in degree one.
(A standard graded $\kk$-algebra is also called a \emph{homogeneous} $\kk$-algebra in other literature.)

\subsection{Background}
The Hilbert series of a (not necessarily commutative) graded $\kk$-algebra is one of the most fundamental invariants in the study of graded $\kk$-algebras.
It provides a concise analytic expression encoding the dimensions of the homogeneous components and thus captures the growth and structural features of the $\kk$-algebra as a $\kk$-vector space. 
It is a classical and well-known fact that if a standard graded $\kk$-algebra is Cohen-Macaulay, then the numerator polynomial of its Hilbert series, called the \textit{$h$-polynomial}, has all positive coefficients.
This fact is one of the cornerstones in graded commutative algebra.
Furthermore, several algebraic properties can be effectively characterized in terms of the Hilbert series. 
A prominent theorem is Stanley's characterization of the Gorenstein property via the symmetry of the Hilbert series (\cite{S78}), 
which has inspired subsequent studies establishing necessary conditions for generalized notions of Gorensteinness, 
such as almost Gorensteinness (\cite{H}), nearly Gorensteinness (\cite{M}), and Artin-Schelter Gorensteinness (\cite{JZ}). 

It is also known that when a standard graded $\kk$-algebra is a complete intersection, its Hilbert series takes a particularly simple form.
In fact, for a positive integer $m$, let \[\Psi_m(t)=1+t+\cdots+t^{m-1}.\]
Then the $h$-polynomial of a standard graded complete intersection is a product of finitely many polynomials $\Psi_m(t)$ (see Proposition~\ref{prop:ci}). 
In addition, the Hilbert series also serves as an essential tool in investigating noncommutative graded complete intersections (see \cite{KKZ, KKZ2}).

\subsection{Cyclotomic standard graded \texorpdfstring{$\kk$}{k}-algebras and polynomials of type CI}

In recent years, increasing attention has been directed to the finer analytic behavior of the Hilbert series of standard graded $\kk$-algebras, particularly the roots of their $h$-polynomials. 
These roots encode subtle structural information and sometimes admit striking combinatorial interpretations 
(see, e.g., the survey \cite{Br} for combinatorial consequences of the real-rootedness of the numerator polynomials of generating series). 

Motivated by this perspective, we focus on the following notion for standard graded $\kk$-algebras:

\begin{dfn}[cf. \cite{BD, KKZ}]
A standard graded $\kk$-algebra $R$ is called \textit{cyclotomic} if all roots of its $h$-polynomial $h_R(t)$ lie on the unit circle in the complex plane. 
\end{dfn}

Since $R$ is standard graded, we have $h_R(0)=1$ and $h_R(1)>0$.
Thus, if all roots of $h_R(t)$ lie on the unit circle, then $h_R(t)$ is monic.
Monic polynomials with integer coefficients whose roots all lie on the unit circle are sometimes called \textit{Kronecker polynomials}.
It is known that every such polynomial can be expressed as a product of cyclotomic polynomials.
Hence, the $h$-polynomial of a cyclotomic $\kk$-algebra is a product of cyclotomic polynomials, although it need not be a cyclotomic polynomial itself.

We say that a polynomial is \textit{of type CI} if it can be written as a product of $\Psi_m(t)$'s. 
In particular, the $h$-polynomials of standard graded complete intersections are of type CI. 
However, the converse does not hold in general; see, for example, \cite[Example~3.7]{S78}. 
Borz\`i and D'Al\`i showed that the converse does hold when $R$ is Koszul~\cite[Theorem~3.9]{BD}. 
Moreover, it was proved that a cyclotomic polynomial $\Phi_m(t)$ coincides with $h_R(t)$ of some standard graded $\kk$-algebra $R$ if and only if $m$ is prime and $R$ is a hypersurface~\cite[Theorem~4.3]{BD}.
Despite these results, in general, one cannot expect the ring-theoretic structure of $R$ to be determined by the cyclotomicity of its $h$-polynomial. 

For a standard graded $\kk$-algebra $R$, the relations among the above notions can be summarized as follows:
\[
\begin{tikzpicture}[
  node distance=6mm and 6mm,
  box/.style={
    draw,
    rectangle,
    align=center,
    inner xsep=5pt,
    inner ysep=5pt
  }
]

\node[box] (A) {
  $R$ is a complete intersection
};

\node[box, below=of A] (B) {
  $h_R(t)$ is of type CI, i.e.,\\[0.5ex]
  $h_R(t)=\prod_j \Psi_{m_j}(t)$
};

\node[box, below=of B] (C) {
  $R$ is cyclotomic, i.e.,\\[0.5ex]
  all roots of $h_R(t)$ lie on the unit circle
};

\node[box, right=18mm of C] (D) {
  $h_R(t)=\prod_j \Phi_{n_j}(t)$\\[0.5ex]
  (a product of cyclotomic polynomials)
};

\node at ($(A.south)!0.5!(B.north)$)
  {$\Downarrow$};

\node[
  right=-15mm of B.north east,
  yshift=3mm,
  font=\small
] {($\Uparrow$ is known in special cases; see \cite{BD})};

\node at ($(B.south)!0.5!(C.north)$)
  {$\Downarrow$};

\node at ($(C.east)!0.5!(D.west)$)
  {$\Longleftrightarrow$};

\end{tikzpicture}
\]
This leads naturally to the following question.

\begin{q}
Which polynomials can occur as the $h$-polynomials of cyclotomic standard graded
$\kk$-algebras, other than those of type CI?
\end{q}

For example, \cite[Theorem~4.3]{BD} indicates that a single cyclotomic polynomial $\Phi_m(t)$ is rarely equal to the $h$-polynomial of a standard graded $\kk$-algebra.

A related perspective appears in Ehrhart theory. 
Braun and Liu investigated lattice polytopes whose $h^*$-polynomials have all roots on the unit circle \cite{BL}. 
Since Ehrhart rings provide a natural source of standard graded Cohen--Macaulay algebras when the underlying lattice polytope has the integer decomposition property, 
and the $h$-polynomial of the Ehrhart ring coincides with the $h^*$-polynomial of the corresponding polytope, 
their work provides a related geometric perspective on the questions considered in this paper. 

\subsection{Main Results}

The aim of this paper is to make a contribution to this question.
Our main result is stated as follows.

\begin{thm}[Main Result]\label{thm:main}
{\em (1)} The $h$-polynomial $h_R(t)$ of any cyclotomic standard graded $\kk$-algebra $R$ is of type CI if $h_R(1) \in \{1,4,6\}$ or $h_R(1)$ is prime. 

\noindent
{\em (2)} There exists a cyclotomic standard graded $\kk$-algebra $R$ whose $h$-polynomial $h_R(t)$ is not of type CI and satisfies $h_R(1)=n$ if $n$ is a composite number at least $8$.
\end{thm}

Note that $h_R(1)$ is nothing but the multiplicity of $R$ (see \cite[Proposition 4.1.9]{BH}), which is one of the important invariants of graded modules.

\subsection{Structure of the paper}
We briefly describe the structure of this paper. 
In Section~\ref{sec:pre}, we fix our notation (Hilbert series, $h$-polynomials, simplicial complexes, etc.), 
review fundamental notions (cyclotomic polynomials, Stanley-Reisner theory, and related topics), 
and recall classical results, namely Macaulay's theorem and the $g$-theorem, which will be used in the proof of Theorem~\ref{thm:main}. 
Sections~\ref{sec:(1)} and~\ref{sec:(2)} are devoted to the proofs of parts~(1) and~(2) of Theorem~\ref{thm:main}, respectively.

\section*{Acknowledgments}
We are grateful to the referee for carefully reading the manuscript and for several insightful observations and valuable suggestions that greatly helped us improve the manuscript.
The first author was supported by JSPS KAKENHI Grant Numbers JP24K00521 and JP24K00534. 
The second author was supported by JSPS KAKENHI Grant Number JP22K03222.

\bigskip

\section{Preliminaries}\label{sec:pre}

In this section, we prepare several preliminary notions and results used throughout this paper.

\subsection{Hilbert series of standard graded \texorpdfstring{$\kk$}{k}-algebras}

Let $R = \bigoplus_{n \ge 0} R_n$ be a standard graded (commutative) $\kk$-algebra.
For $n \ge 0$, define $H(R,n)=\dim_\kk R_n$. We call the function $H(R,-) : \ZZ_{\ge 0} \to \ZZ_{\ge 0}$ the \textit{Hilbert function} of $R$. 
Let $\Hilb(R,t)$ denote the \textit{Hilbert series} of $R$, that is, $\Hilb(R,t) = \sum_{n \ge 0} H(R,n)\, t^n$. 
Then $\Hilb(R,t)$ can be written in the form 
\[\Hilb(R,t) = \frac{h_R(t)}{(1-t)^d},\]
where $d$ is the Krull dimension of $R$ and $h_R(t)$ is a polynomial with integer coefficients satisfying $h_R(1) > 0$ (see, e.g., \cite[Corollary~4.1.8]{BH}). 
We call the polynomial $h_R(t)$ appearing in the numerator the \textit{$h$-polynomial} of $R$. 
It is known that $h_R(1)$ coincides with the multiplicity of $R$ (\cite[Proposition 4.1.9]{BH}).

Note that the Hilbert function of $R$ can be recovered from its $h$-polynomial and Krull dimension. 
Indeed, if $h_R(t) = \sum_{i \ge 0} h_i t^i$ and $d$ is the Krull dimension of $R$, then 
\begin{align}\label{eq:hilb}
H(R,n) = \sum_{i \ge 0} h_i \binom{d-1+n-i}{n-i}.
\end{align}
For more details on basic facts about Hilbert functions and Hilbert series, see \cite[Section~4]{BH}.

\medskip

Regarding the $h$-polynomials of standard graded complete intersections, we have the following result.

\begin{prop}[{cf.~\cite[Corollary~3.4]{S78}}]\label{prop:ci}
Let $R$ be a standard graded $\kk$-algebra. 
If $R$ is a complete intersection, then the $h$-polynomial of $R$ is of type~CI, that is, $h_R(t)=\prod_{i=1}^r\Psi_{m_i}(t)$ for some non-negative integer $r$ and positive integers $m_1, \dots, m_r$.
\end{prop}

\subsection{Cyclotomic polynomials}

Next, we recall some fundamental facts about cyclotomic polynomials. 
For an introduction to cyclotomic polynomials, see, e.g., \cite{L}. 

For a positive integer $m$, let $\Phi_m(t)$ denote the \textit{$m$-th cyclotomic polynomial}, i.e.,
\[\Phi_m(t) = \prod_{\substack{1 \le k \le m \\ \gcd(k,m)=1}}\left(t - e^{2\pi k\sqrt{-1}/m}\right).\]
We collect some well-known properties of cyclotomic polynomials: 
\begin{enumerate}
\item $\Phi_m(t)$ is an irreducible polynomial in $\ZZ[t]$ of degree $\varphi(m)$, where \[\varphi(m) = |\{\, k \in \{1,\ldots,m\} : \gcd(k,m)=1 \}|.\] 
\item For $m\ge 2$, $\Phi_m(t)$ is palindromic, i.e., $t^{\varphi(m)} \Phi_m(t^{-1}) = \Phi_m(t)$. 
\item For $m \ge 2$, 
\begin{align}\label{eq:prime}
\Phi_m(1) = 
\begin{cases}
p & \text{if $m$ is a power of a prime $p$,} \\
1 & \text{otherwise.}
\end{cases}
\end{align}
\item For a prime power $p^c$, 
\begin{align}\label{eq:power}
\Phi_{p^c}(t) = \Phi_p(t^{p^{c-1}}).
\end{align}
\item For any $m$, 
\begin{align}\label{eq:psi-factorization}
\Psi_m(t) = \prod_{\substack{1 < d \le m \\ d \mid m}} \Phi_d(t).
\end{align}
In particular, $\Psi_m(t) = \Phi_m(t)$ if $m$ is prime. 
\item Any monic polynomial with integer coefficients whose roots all lie on the unit circle in the complex plane can be expressed as a product of cyclotomic polynomials.
\end{enumerate}

\begin{ex}\label{ex:small}
We list the cyclotomic polynomials of small degrees:
\begin{align*}
\deg 1 : \;& \Phi_1(t) = t - 1,  \quad \Phi_2(t) = 1 + t; \\[2mm]
\deg 2 : \;& \Phi_3(t) = 1 + t + t^2, \quad 
\Phi_4(t) = \Phi_2(t^2) = 1 + t^2, \quad 
\Phi_6(t) = 1 - t + t^2; \\[2mm]
\deg 4 : \;& \Phi_5(t) = 1 + t + t^2 + t^3 + t^4, \quad 
\Phi_8(t) = \Phi_2(t^4) = 1 + t^4, \\ 
& \Phi_{10}(t) = 1 - t + t^2 - t^3 + t^4, \quad 
\Phi_{12}(t) = \Phi_6(t^2) = 1 - t^2 + t^4.
\end{align*}
\end{ex}

The following lemma gives a simple way to check that a polynomial is not of type CI.

\begin{lem}\label{l.nonCI}
Let $f(t)$ be a polynomial in $\ZZ[t]$.
If $\Phi_m(t)$ divides $f(t)$, while
$\Phi_d(t)$ does not divide $f(t)$ for some divisor $d$ of $m$ with $d>1$, then $f(t)$ is not of type CI.
\end{lem}

\begin{proof}
Suppose that $f(t)$ is of type CI. Then $f(t)=\prod_i \Psi_{n_i}(t)$.
Since $\Phi_m(t)$ divides $f(t)$, it divides $\Psi_{n_i}(t)$ for some $i$. By \eqref{eq:psi-factorization}, $\Psi_{n_i}(t)$ is also divisible by $\Phi_d(t)$ for any divisor $d>1$ of $m$,
a contradiction.
\end{proof}

\subsection{Simplicial complexes and their \texorpdfstring{$h$}{h}-polynomials}
Here, we review simplicial complexes and their $h$-polynomials. 

For a given non-empty finite set $V$, let $2^V$ denote the collection of all subsets of $V$ (including the empty set). 
A subset $\Delta \subset 2^V$ is called a \textit{simplicial complex} if it satisfies the following conditions:
\begin{itemize}
\item $\{v\} \in \Delta$ for every $v \in V$;
\item if $F \in \Delta$ and $F' \subset F$, then $F' \in \Delta$. 
\end{itemize}
Note that $ \varnothing \in \Delta$ for every simplicial complex $\Delta$. 

An element $F \in \Delta$ is called an \textit{$i$-face} if $|F| = i + 1$.  In particular, $\varnothing$ is the unique $(-1)$-face.  
We define the \textit{dimension} of $\Delta$ by 
\[\dim \Delta = \max \{\, |F| - 1 : F \in \Delta \,\}.\]
We then set 
\[f_i(\Delta) = |\{\, F \in \Delta : F \text{ is an $i$-face} \,\}| \quad \text{and} \quad f(\Delta) = (f_{-1}(\Delta), f_0(\Delta), \ldots, f_{d-1}(\Delta)),\] 
where $d - 1 = \dim \Delta$. The vector $f(\Delta)$ is called the \textit{$f$-vector} of $\Delta$.  
The \textit{$h$-polynomial} of $\Delta$ is defined by 
\begin{align}\label{eq:fh}
\sum_{i=0}^d h_i(\Delta) t^{d-i}  = \sum_{i=0}^d f_{i-1}(\Delta) (t-1)^{d-i}.
\end{align}
(Note that, when $d$ is fixed, the $h$-polynomial carries the same information as the $f$-vector via~\eqref{eq:fh}.)

\medskip

For a finite set $V$, let $\Delta_V = 2^V$.  
Clearly, $\Delta_V$ is a simplicial complex of dimension $|V|-1$, which is usually called the \textit{simplex} on $V$. 
We define the \textit{boundary complex} of a simplex by $\partial \Delta_V = \Delta_V \setminus \{V\}$. 
We write $\Delta_d$ (resp. $\partial \Delta_d$) for $\Delta_V$ (resp. $\partial \Delta_V$) when $|V| = d$. 
Then, as an easy exercise, one obtains 
\begin{align}\label{eq:exercise}
h_{\Delta_d}(t) = 1 \quad \text{and} \quad h_{\partial \Delta_d}(t) = \Psi_d(t). 
\end{align}

\subsection{Stanley-Reisner rings of simplicial complexes}

We now provide a brief introduction to the theory of Stanley-Reisner rings (cf.~\cite[Section~5]{BH}). 
This theory plays a crucial role in analyzing the Hilbert series of standard graded $\kk$-algebras. 

For a simplicial complex $\Delta$ on $V = \{1, \ldots, d\}$, define
\[\kk[\Delta] := \kk[x_1, \ldots, x_d] / I_\Delta, \quad \text{where} \quad I_\Delta = \left( \prod_{i \in F} x_i \; : \; F \subset V, \, F \notin \Delta \right).\] 
The standard graded $\kk$-algebra $\kk[\Delta]$ (resp. the ideal $I_\Delta$) is called the \textit{Stanley-Reisner ring} (resp. the \textit{Stanley-Reisner ideal}) of $\Delta$.  
The following facts are well known: 
\begin{itemize}
\item The Krull dimension of $\kk[\Delta]$ is equal to $\dim \Delta + 1$.
\item The $h$-polynomial of $\kk[\Delta]$ coincides with the $h$-polynomial of $\Delta$.
\item The Cohen-Macaulayness and Gorensteinness of $\kk[\Delta]$ can be characterized in terms of the reduced homology of $\Delta$. 
\end{itemize}

From~\eqref{eq:exercise}, we obtain the following equalities, which will be used several times in the sequel:
\[\Hilb(\kk[\Delta_d], t) = \frac{1}{(1-t)^d} \quad \text{and} \quad \Hilb(\kk[\partial \Delta_d], t) = \frac{\Psi_d(t)}{(1-t)^{d-1}}.\]

\subsection{Macaulay's theorem and \texorpdfstring{$g$}{g}-theorem}

Finally, let us review Macaulay's theorem and the $g$-theorem. To state them, we first recall the notion of a binomial sum expansion. 

Given a positive integer $a$, there exists a unique expression
\begin{align}\label{eq:binom}
a = \binom{k_d}{d} + \binom{k_{d-1}}{d-1} + \cdots + \binom{k_{d'}}{d'},
\end{align}
for any fixed positive integer $d$, where $d'$ is some positive integer satisfying $d \ge d'$ and $k_d > k_{d-1} > \cdots > k_{d'} \ge d'$ (see \cite[Lemma~4.2.6]{BH}). 
We call the expression~\eqref{eq:binom} the \textit{binomial sum expression} of $a$ with respect to $d$. We then define
\[a^{\langle d \rangle} = \binom{k_{d}+1}{d+1} + \binom{k_{d-1}+1}{d} + \cdots + \binom{k_{d'}+1}{d'+1}.\]

\begin{thm}[{Macaulay's theorem, cf.~\cite[Theorem~4.2.10]{BH} and \cite[Chapter~II, Theorem~2.2]{Sbook}}]
Let $h : \ZZ_{\ge 0} \to \ZZ_{\ge 0}$ be a function with $h(0) = 1$. 
Then there exists a standard graded $\kk$-algebra $R$ satisfying $H(R,n) = h(n)$ for all $n \ge 0$ if and only if $h(n+1) \le h(n)^{\langle n \rangle}$ for every $n > 0$.
\end{thm}

\begin{ex}\label{ex:111}
We illustrate how to use Macaulay's theorem. 
Let $h(t) = \sum_{i=0}^s h_i t^i$ be a polynomial such that $h_0 = h_1 = \cdots = h_{a-1} = 1$ and $h_a > 1$ for some $1 < a \le s$. 
Then $h(t)$ cannot be the $h$-polynomial of any standard graded $\kk$-algebra.  
Indeed, suppose there exists such an algebra $R$, and let $d$ denote its Krull dimension. Then 
\begin{align*}
H(R,a-1)
&= \sum_{i \ge 0} h_i \binom{d-1+a-1-i}{a-1-i}  = \binom{d+a-2}{a-1} + \binom{d+a-3}{a-2} + \cdots + \binom{d-1}{0} \\
&= \binom{d+a-1}{a-1}, \\
\text{while} \quad 
H(R,a)&= \sum_{i \ge 0} h_i \binom{d-1+a-i}{a-i}  = \binom{d+a-1}{a} + \cdots + \binom{d}{1} + h_a \binom{d-1}{0} \\
&= \binom{d+a}{a} + h_a - 1. 
\end{align*}
Since $H(R,a-1)^{\langle a-1 \rangle} = \binom{d+a}{a}$, we obtain
\[H(R,a-1)^{\langle a-1 \rangle} - H(R,a) = 1- h_a < 0,\]
a contradiction to Macaulay's theorem.
\end{ex}

\begin{ex}\label{ex:modified}
We next describe another polynomial that cannot occur as the $h$-polynomial of any standard graded $\kk$-algebra. 
Consider $h(t) = 1 + 2t + t^2 + 3t^5$, and suppose that there exists a standard graded $\kk$-algebra $R$ of Krull dimension $d$ with this $h$-polynomial. 
Then, from~\eqref{eq:hilb}, we have
\begin{align*}
H(R,4) &= \binom{d+3}{4} + 2\binom{d+2}{3} + \binom{d+1}{2} \quad\text{and} \\
H(R,5) &= \binom{d+4}{5} + 2\binom{d+3}{4} + \binom{d+2}{3} + 3.
\end{align*}
Without loss of generality, we may assume that $d \geq 3$.
By a direct computation, the binomial sum expansion of $H(R,4)$ is
\begin{align*}
\binom{d+3}{4} + 2\binom{d+2}{3} + \binom{d+1}{2} = \binom{d+4}{4} + \binom{d+1}{3} + \binom{d-1}{2} + \binom{d-2}{1}. 
\end{align*}
Hence, 
\[H(R,4)^{\langle 4 \rangle}  = \binom{d+5}{5} + \binom{d+2}{4} + \binom{d}{3} + \binom{d-1}{2}.\]
In this case, one can show that
\[H(R,4)^{\langle 4 \rangle} - H(R,5) = -1,\]
which contradicts Macaulay's theorem. However, establishing this fact is far from trivial. 

As illustrated by this example, the computation of binomial expansions is generally rather involved, 
and hence it is not easy to apply Macaulay's theorem to determine whether a given polynomial can be realized as an $h$-polynomial.
\end{ex}

\medskip

For a simplicial polytope $P$ with vertex set $V$, the collection
\[\{ F \subset V : \mathrm{conv}(F) \text{ is a proper face of } P \}\]
forms a simplicial complex on $V$, called the \textit{boundary complex} of $P$. 
Note that the boundary complex of a simplex of dimension $d-1$ is precisely $\partial \Delta_d$. 

A characterization of the $h$-polynomials of boundary complexes of simplicial polytopes is known as the \textit{$g$-theorem}, 
one of the most celebrated results in Stanley-Reisner theory.

\begin{thm}[{$g$-theorem, cf.~\cite[Chapter~III, Theorem~1.1]{Sbook}}]\label{thm:g}
Let $(h_0, h_1, \ldots, h_d) \in \ZZ^{d+1}$. 
Then there exists a simplicial polytope $P$ of dimension $d$ whose boundary complex has the $h$-polynomial $\sum_{i=0}^d h_i t^i$ 
if and only if the following two conditions are satisfied:
\begin{enumerate}
\item $h_i = h_{d-i}$ for $0 \le i \le \lfloor d/2 \rfloor$, and
\item $h_0 = 1$, $h_1-h_0\geq 0$, and $0 \le h_{i+1} - h_i \le (h_i - h_{i-1})^{\langle i \rangle}$ for $1 \le i \le \lfloor d/2 \rfloor - 1$.
\end{enumerate}
\end{thm}

\bigskip

\section{Proof of Theorem~\ref{thm:main} (1)}\label{sec:(1)}
This section is devoted to proving Theorem~\ref{thm:main} (1). 

\begin{lem}\label{lem:D-2}
Let $\Delta$ be a simplicial complex of dimension $D-1$ and let $f_{D-1}(\Delta)=a$. Then 
\[f_{D-2}(\Delta) \geq \begin{cases}
a D - \binom{a}{2} &\text{if }a \leq D+1, \\
\binom{D+1}{2}&\text{if } a > D+1. 
\end{cases}\]
\end{lem}
\begin{proof}
If $a=1$, since there are at least $D$ faces of dimension $D-2$, we get $f_{D-2}(\Delta) \geq D$. 

Let $a \leq D+1$ and let $F_1,\ldots,F_a$ be the $(D-1)$-faces of $\Delta$. 
Let \[\Delta'=\langle F_1,\ldots,F_{a-1} \rangle = \{G \in \Delta : G \subset F_i \text{ for some }i\}.\] 
Since $f_{D-2}(\Delta' \cap F_a) \leq a -1$, we see by induction on $a$ that 
\begin{align*}
f_{D-2}(\Delta) \geq f_{D-2}(\Delta') + D - f_{D-2}(\Delta' \cap F_a) \geq (a-1)D -\binom{a-1}{2} + D - (a-1) =a D - \binom{a}{2}. 
\end{align*}

If $a > D+1$, let $\Delta''$ be a subcomplex of $\Delta$ consisting of some $(D-1)$-faces $F_1,\ldots,F_{D+1} \in \Delta$. 
Then $f_{D-2}(\Delta'') \geq (D+1)D - \binom {D+1}{2}=\binom{D+1}{2}$ by the first case. Thus, 
\[f_{D-2}(\Delta) \geq f_{D-2}(\Delta'') \geq \binom{D+1}{2},\]
as required. 
\end{proof}

The following proposition plays the essential role for the proof of Theorem~\ref{thm:main} (1). 
\begin{prop}\label{prop}
Let $R$ be a standard graded $\kk$-algebra whose $h$-polynomial is palindromic. 
Then we have $\deg(h_R(t)) \leq h_R(1) -1$. 
\end{prop}
\begin{proof}
We employ a technique similar to that used in the proof of \cite[Lemma~4.6]{BD}. 

Let $h_R(t)=\sum_{i=0}^s h_it^i$, where $s$ is the degree of $h_R(t)$. Since $h_R(t)$ is palindromic, we have $h_i=h_{s-i}$ for $i=0,1,\ldots,s$. We set $a=h_R(1)$. 

Let $R \cong S/I$, where $S$ is a polynomial ring and $I$ is an ideal of $S$. 
Since the Hilbert series does not change when passing to the initial ideal of $I$, we can assume without loss of generality that $I$ is a monomial ideal. 
Moreover, since polarization preserves $h_R(t)$, we can further assume that $I$ is a squarefree monomial ideal. 
Hence, $I$ is the Stanley-Reisner ideal of some simplicial complex $\Delta$.
Let $\dim \Delta =D-1$. Then $s \leq D$ and $h_\Delta(t)=h_R(t)$, that is, $h_i(\Delta)=h_i$ for every $0 \leq i \leq s$.
By \eqref{eq:fh}, we see that 
\begin{align*}
f_{D-1}(\Delta)&=h_R(1)=a, \text{ and}\\
f_{D-2}(\Delta)&=\sum_{i=0}^s(D-i)h_i(\Delta)=(D-s)\sum_{i=0}^sh_i(\Delta)+\frac{s}{2}\sum_{i=0}^sh_i(\Delta)=\left(D-\frac{s}{2}\right)a. 
\end{align*}

Suppose that $a \leq D+1$.  Then, by Lemma~\ref{lem:D-2}, we have 
\begin{align*}
f_{D-2}(\Delta)=\left(D-\frac{s}{2}\right)a \geq a D - \binom{a}{2}
\end{align*}
and hence $s \leq a-1$. This is the desired conclusion. 

Suppose that $a > D+1$. Then, by Lemma~\ref{lem:D-2}, we have
\begin{align*}
f_{D-2}(\Delta) =\left(D-\frac{s}{2}\right)a \geq \binom{D+1}{2}.
\end{align*}
This implies $2Da - D^2 - D \geq sa$, and hence, $a(a-s) \geq (D-a)^2 + D > 0$.
Therefore, $a > s$, i.e., $s \leq a-1$. 
\end{proof}

By using Proposition~\ref{prop}, we can show the following two propositions, which give a proof of Theorem~\ref{thm:main} (1). 
\begin{prop}\label{prop:1prime}
Let $R$ be cyclotomic and assume that $h_R(1)$ is $1$ or a prime $p$. Then $R$ is regular or 
a hypersurface defined by a homogeneous polynomial of degree $p$, respectively. In particular, $h_R(t)$ is of type CI. 
\end{prop}
\begin{proof}
In the case of $h_R(1)=1$, Proposition~\ref{prop} directly implies $h_R(t)=1$, i.e., $R$ is regular.  

If $h_R(1)=p$ is prime, then $h_R(t)$ is divisible by $\Phi_p(t^i)$ for some $i \geq 1$. 
By Proposition~\ref{prop}, we see that $\deg (h_R(t)) \leq h_R(1)-1$. Hence, $i=1$ and $h_R(t)$ must be $\Phi_p(t)$ itself. 
Then \cite[Theorem 4.3]{BD} implies that $R$ is a hypersurface. 
\end{proof}

\begin{prop}
Let $f(t)$ be a monic polynomial with integer coefficients
whose roots all lie on the unit circle. Assume that $\deg (f(t)) \leq f(1)-1$. \\
{\em (1)} If $f(1)=4$, then $f(t)$ is of type CI. \\
{\em (2)} If $f(1)=6$ and $f(t)=h_R(t)$ for some standard graded $\kk$-algebra $R$, then $f(t)$ is of type CI. 
\end{prop}
\begin{proof}
(1) Consider all possible polynomials $f(t)$ with $f(1)=4$ and $\deg(f(t)) \leq 3$ constructed as a product of cyclotomic polynomials. 
(See Example~\ref{ex:small} for cyclotomic polynomials of small degrees.) 
Then such polynomials are $\Phi_2(t)^2=(1+t)^2$ and $\Phi_2(t)\Phi_4(t)=1+t+t^2+t^3$.

(2) Similarly, all possible polynomials $f(t)$ with $f(1)=6$ and $\deg (f(t)) \leq 5$ are 
\[\Phi_2(t)\Phi_3(t)=1+2t+2t^2+t^3, \; \Phi_2(t)\Phi_3(t)\Phi_6(t)=\Psi_6(t), \text{ and }\Phi_3(t)\Phi_4(t)=1+t+2t^2+t^3+t^4.\] 
The polynomials $\Phi_2(t)\Phi_3(t)$ and $\Psi_6(t)$ are of type CI, 
while we know that $1+t+2t^2+t^3+t^4$ cannot be the $h$-polynomial of any standard graded $\kk$-algebra (see Example~\ref{ex:111}). 
\end{proof}

\bigskip

\section{Proof of Theorem~\ref{thm:main} (2)}\label{sec:(2)}
The goal of this section is to give a proof of Theorem~\ref{thm:main} (2). 
Specifically, for every composite integer $n \geq 8$, we construct a cyclotomic standard graded $\kk$-algebra $R$ such that $h_R(t)$ is not of type CI and $h_R(1)=n$. 
Such $\kk$-algebras are realized as Stanley-Reisner rings of certain simplicial complexes. 

We recall methods for constructing a simplicial complex from two simplicial complexes. 
\begin{itemize}
\item Consider two simplicial complexes $\Delta_1$ and $\Delta_2$ on $V_1$ and $V_2$, respectively, with $V_1 \cap V_2 \neq \varnothing$. 
Then $\Delta = \Delta_1 \cup \Delta_2$ and $\Delta'=\Delta_1 \cap \Delta_2$ are simplicial complexes on $V_1 \cup V_2$ and $V_1 \cap V_2$, respectively. 
Thus, by the inclusion-exclusion principle, we see that 
\begin{align}\label{eq:ie}\Hilb(\kk[\Delta],t)=\Hilb(\kk[\Delta_1],t)+\Hilb(\kk[\Delta_2],t)-\Hilb(\kk[\Delta'],t).\end{align}
\item Consider two simplicial complexes $\Delta_1$ and $\Delta_2$ on $V_1$ and $V_2$, respectively, and assume that $V_1 \cap V_2 =  \varnothing$. 
Then the \textit{join} of $\Delta_1$ and $\Delta_2$, denoted by $\Delta_1 * \Delta_2$, is the simplicial complex on $V_1 \cup V_2$ defined by 
\[\{F_1 \cup F_2 : F_1 \in \Delta_1, F_2 \in \Delta_2\}.\]
Note that $\dim (\Delta_1 * \cdots * \Delta_\ell) = \dim \Delta_1 + \cdots + \dim \Delta_\ell + \ell-1$. 
It is well known that
\[h_{\Delta_1*\Delta_2}(t)=h_{\Delta_1}(t) \cdot h_{\Delta_2}(t). \]
\end{itemize}

\bigskip

The following examples correspond to the cases $n=8,9,12$, respectively. 
\begin{ex}\label{ex_8912}
(1) (The case $h_R(1)=8$) Let $R$ be the Stanley-Reisner ring of the following simplicial complex $\Delta$ on $V \cup \{v'\}$, where $|V|=8$ and $v' \notin V$. 
We construct it as the union of $\partial\Delta_{|V|}$ (of dimension $6$) and 
$\partial\Delta_2 * \partial\Delta_3 * \underbrace{\Delta_1 * \{v'\}}_{\cong \Delta_2}$ (of dimension $4$). 
Then their intersection is $\partial\Delta_2 * \partial\Delta_3 * \Delta_1$ (of dimension $3$). 
Note that $\partial\Delta_2 * \partial\Delta_3 * \Delta_1$ can be found inside of $\partial\Delta_V = 2^V \setminus \{V\}$. 
Hence, it follows from \eqref{eq:ie} that \begin{align*}
\Hilb(\kk[\Delta],t)&=\frac{h_{\partial	\Delta_8}(t)}{(1-t)^7}+\frac{h_{\partial\Delta_2 * \partial\Delta_3 * \Delta_2}(t)}{(1-t)^5}-\frac{h_{\partial\Delta_2 * \partial\Delta_3 * \Delta_1}(t)}{(1-t)^4} \\
&=\frac{\Psi_8(t)}{(1-t)^7} + \frac{\Psi_2(t)\Psi_3(t)}{(1-t)^5} - \frac{\Psi_2(t)\Psi_3(t)}{(1-t)^4} \\
&=\frac{1+t+\cdots+t^7+(1-(1-t))(1-t)^2(1+t)(1+t+t^2)}{(1-t)^7} \\
&=\frac{1+2t+t^2+t^5+2t^6+t^7}{(1-t)^7} =\frac{\Phi_2(t)^3\Phi_{10}(t)}{(1-t)^7}. 
\end{align*}
Explicitly,
\[\kk[\Delta] \cong \kk[x_1,\ldots,x_9]/(x_1x_2x_3x_4x_5x_6x_7x_8,x_1x_2x_9,x_3x_4x_5x_9,x_7x_9,x_8x_9).\]

\noindent
(2) (The case $h_R(1)=9$) Let $R$ be the Stanley-Reisner ring of $\Delta$ which can be constructed as the union of $\partial \Delta_9$ (of dimension $7$) and 
$\partial \Delta_2*\partial \Delta_2* \partial \Delta_3 * \Delta_1 * \{v'\}$ (of dimension $5$). 
Then their intersection is $\partial \Delta_2*\partial \Delta_2* \partial \Delta_3 * \Delta_1$ (of dimension $4$). Thus, 
\begin{align*}
h_R(t)&=1+t+\cdots+t^8 + (1-t)^2(1+t)^2(1+t+t^2) - (1-t)^3(1+t)^2(1+t+t^2) \\
&=1+2t+2t^2-t^4+2t^6+2t^7+t^8 \\
&=\Phi_3(t)^2\Phi_{12}(t). 
\end{align*}
Explicitly, 
\begin{align*}
\kk[\Delta] \cong \kk[x_1,\ldots,x_{10}]/(x_1x_2x_3x_4x_5x_6x_7x_8x_9,x_1x_2x_{10},x_3x_4x_{10},x_5x_6x_7x_{10},x_9x_{10}). 
\end{align*}

\noindent
(3) (The case $h_R(1)=12$) Let $R$ be the Stanley-Reisner ring of $\Delta$ which can be constructed as the union of 
$\partial \Delta_{12}$ (of dimension $10$) and $\partial \Delta_2 * \partial \Delta_2 * \partial \Delta_3 * \partial \Delta_4 * \Delta_1 * \{v'\}$ (of dimension $8$). 
Then their intersection is $\partial \Delta_2 * \partial \Delta_2 * \partial \Delta_3 * \partial \Delta_4 * \Delta_1$ (of dimension $7$). 
Thus, \begin{align*}
h_R(t)&=1+t+\cdots+t^{11} + (1-t)^2\Psi_2(t)^2\Psi_3(t)\Psi_4(t) - (1-t)^3\Psi_2(t)^2\Psi_3(t)\Psi_4(t) \\
&=1+2t+3t^2+2t^3-2t^5-2t^6+2t^8+3t^9+2t^{10}+t^{11} \\
&=\Psi_3(t)\Psi_4(t)\Phi_{18}(t). 
\end{align*}
Explicitly,
\begin{align*}
\kk[\Delta] \cong 
\kk[x_1,\ldots,x_{13}]/(x_1x_2x_3x_4x_5x_6x_7x_8x_9x_{10}x_{11}x_{12},x_1x_2x_{13},x_3x_4x_{13}, x_5x_6x_7x_{13},x_8x_9x_{10}x_{11}x_{13}).
\end{align*}
\end{ex}

We prove the following theorem, which implies Theorem~\ref{thm:main} (2). 
\begin{thm}\label{thm:p=3}
{\em (1)} Let $q \ge 5$ be an odd integer.
Then there exists a cyclotomic standard graded $\kk$-algebra $R$ such that $h_R(t)$ is not of type CI with $h_R(1)=2q$. 

\noindent
{\em (2)} Let $q \ge 3$ be an odd integer.
Then there exists a cyclotomic standard graded $\kk$-algebra $R$ such that $h_R(t)$ is not of type CI with $h_R(1)=3q$. 

\noindent
{\em (3)} Let $5 \le p \le q$ be odd integers.
Then there exists a cyclotomic standard graded $\kk$-algebra $R$ such that $h_R(t)$ is not of type CI with $h_R(1)=pq$. 
\end{thm}

In preparation for the proofs of Theorem~\ref{thm:p=3} (1) and (2), we state the following lemma. 
\begin{lem}\label{lem:decomp} 
{\em (1)} Let $q \geq 5$ be an odd integer. Then the following holds: 
\[\Psi_2(t) \Psi_q(t) \Phi_6(t^{(q-1)/2})=\Psi_{2q}(t)+t(1-t)^2 \Psi_{(q-3)/2}(t) \Psi_{(q-1)/2}(t) \Psi_q(t).\]

\noindent
{\em (2)} Let $q \geq 3$ be an odd integer. Then the following holds: 
\[\Psi_3(t) \Psi_q(t) \Phi_6(t^{q-1})=\Psi_{3q}(t)+t(1-t)^2 \Psi_2(t) \Psi_{q-2}(t) \Psi_{q-1}(t) \Psi_q(t).\]
\end{lem}
\begin{proof}
(1) We can directly check the following: 
\begin{align*}
(1-t)^2(\Psi_2(t) \Psi_q(t) \Phi_6(t^{(q-1)/2})-\Psi_{2q}(t))&=(1-t^2)(1-t^q)(1-t^{\frac{q-1}{2}}+t^{q-1})-(1-t)(1-t^{2q}) \\
&=t(1-t)^4\Psi_{(q-3)/2}(t)\Psi_{(q-1)/2}(t)\Psi_q(t). 
\end{align*}
This yields the required equation. 

(2) We can directly check the following: 
\begin{align*}
(1-t)^2(\Psi_3(t)\Psi_q(t)\Phi_6(t^{q-1})-\Psi_{3q}(t)) &=(1-t^3)(1-t^q)(1-t^{q-1}+t^{2q-2})-(1-t)(1-t^{3q}) \\
&=t(1-t)^4\Psi_2(t)\Psi_{q-2}(t)\Psi_{q-1}(t)\Psi_q(t).  
\end{align*}
This yields the required equation. 
\end{proof}

We also note that, for every positive integer $m$,
$\Phi_{6m}(t)\mid \Phi_6(t^m)$
and 
$\Phi_{3m}(t) \notmid \Phi_6(t^m)$.
Indeed, if $\zeta$ is a primitive $6m$-th root of unity, then
$\zeta^m$ is a primitive sixth root of unity, while if $\xi$ is a
primitive $3m$-th root of unity, then $\xi^m$ is a primitive third root of unity but not a primitive sixth root of unity.

\begin{proof}[Proofs of Theorem~\ref{thm:p=3} (1) and (2)]
(1) Let $\Delta$ be the simplicial complex constructed as the union of $\partial \Delta_{2q}$ (of dimension $2q-2$) and 
$\partial \Delta_{(q-3)/2} * \partial \Delta_{(q-1)/2} * \partial\Delta_q*\Delta_1 * \{v'\}$ (of dimension $2q-4$), 
whose intersection is $\partial \Delta_{(q-3)/2} * \partial \Delta_{(q-1)/2} * \partial\Delta_q*\Delta_1$ (of dimension $2q-5$). 
(We see that this can be found inside of $\partial \Delta_{2q}$ by considering the number of vertices.) It then follows from \eqref{eq:ie} that  
\begin{align*}
\Hilb(\kk[\Delta],t)&=\frac{h_{\partial \Delta_{2q}}(t)}{(1-t)^{2q-1}}+\frac{h_{\partial \Delta_{(q-3)/2} * \partial \Delta_{(q-1)/2} * \partial\Delta_q}(t)}{(1-t)^{2q-3}} 
-\frac{h_{\partial \Delta_{(q-3)/2} * \partial \Delta_{(q-1)/2} * \partial\Delta_q}(t)}{(1-t)^{2q-4}} \\
&=\frac{\Psi_{2q}(t)+((1-t)^2-(1-t)^3)\Psi_{(q-3)/2}(t)\Psi_{(q-1)/2}(t)\Psi_q(t)}{(1-t)^{2q-1}} \\
&=\frac{\Psi_2(t)\Psi_q(t)\Phi_6(t^{(q-1)/2})}{(1-t)^{2q-1}} \quad(\text{by Lemma~\ref{lem:decomp} (1)}). 
\end{align*}
Clearly, $\kk[\Delta]$ is cyclotomic.
Set $m=(q-1)/2 \geq 2$.
We see that $\Phi_{6m}(t)$ divides $h_{\kk[\Delta]}(t)=\Psi_2(t)\Psi_q(t)\Phi_6(t^m)$. Since $q<3m$, we also see that $\Phi_{3m}(t)$ does not divide $h_{\kk[\Delta]}(t)$.
Hence, by Lemma~\ref{l.nonCI},
$h_{\kk[\Delta]}(t)$ is not of type CI.

(2) Let $\Delta$ be the simplicial complex constructed as the union of $\partial \Delta_{3q}$ (of dimension $3q-2$) and 
$\partial \Delta_2 * \partial \Delta_{q-2} * \partial \Delta_{q-1} * \partial\Delta_q * \Delta_1 * \{v'\}$ (of dimension $3q-4$), 
whose intersection is $\partial \Delta_2 * \partial \Delta_{q-2} * \partial \Delta_{q-1} * \partial\Delta_q * \Delta_1$ (of dimension $3q-5$). Then 
\begin{align*}
\Hilb(\kk[\Delta],t)&=\frac{h_{\partial \Delta_{3q}}(t)}{(1-t)^{3q-1}}+\frac{h_{\partial \Delta_2 * \partial \Delta_{q-2} * \partial \Delta_{q-1} * \partial\Delta_q * \Delta_2}(t)}{(1-t)^{3q-3}} 
-\frac{h_{\partial \Delta_2 * \partial \Delta_{q-2} * \partial \Delta_{q-1} * \partial\Delta_q * \Delta_1}(t)}{(1-t)^{3q-4}} \\
&=\frac{\Psi_{3q}(t)+((1-t)^2-(1-t)^3)\Psi_2(t)\Psi_{q-2}(t)\Psi_{q-1}(t)\Psi_q(t)}{(1-t)^{3q-1}} \\
&=\frac{\Psi_3(t)\Psi_q(t)\Phi_6(t^{q-1})}{(1-t)^{3q-1}} \quad(\text{by Lemma~\ref{lem:decomp} (2)}). 
\end{align*}
Clearly, $\kk[\Delta]$ is cyclotomic.
Set $m=q-1 \geq 2$. 
We see that $\Phi_{6m}(t)$ divides
$h_{\kk[\Delta]}(t)=\Psi_3(t)\Psi_q(t)\Phi_6(t^m)$.
Since $q<3m$, we also see that $\Phi_{3m}(t)$ does not divide
$h_{\kk[\Delta]}(t)$.
Hence, by Lemma~\ref{l.nonCI},
$h_{\kk[\Delta]}(t)$ is not of type CI.
\end{proof}

For the proof of Theorem~\ref{thm:p=3} (3), we introduce several notations and equalities. 
Given odd integers $p,q$ with $3 \leq p \leq q$, let 
\[g_{p,q}(t)=\Psi_p(t)\Psi_q(t)\Phi_6(t^{(p-1)(q-1)/2})-\Psi_{pq}(t) \;\text{ and }\; f_{p,q}(t)=\frac{g_{p,q}(t)}{t(1-t)^2\Psi_{q-1}(t)}.\]
Note that $f_{3,q}(t)=\Psi_2(t)\Psi_{q-2}(t)\Psi_q(t)$ by Lemma~\ref{lem:decomp} (2). 
\begin{lem}\label{lem:fpq}
For any odd integers $p,q$ with $3 \leq p \leq q$, $f_{p,q}(t)$ is a palindromic polynomial in $t$ of degree $(p-1)q-3$. 
\end{lem}
\begin{proof}
First, we verify the polynomiality of $f_{p,q}(t)$.
Namely, it suffices to show that
$g_{p,q}(t)$ (resp.\ its derivative $g_{p,q}'(t)$) has zeros at $0$, $1$, and $e^{2\pi j\sqrt{-1}/(q-1)}$ for all $j=1,\ldots,q-2$ (resp.\ at $1$),
which can be verified straightforwardly.

Next, we confirm the degree of $f_{p,q}(t)$, but it is also straightforward. 
In fact, we have $\deg(g_{p,q}(t))=pq-2$ (note that the subtraction causes cancellation of the leading terms) and $\deg(t(1-t)^2\Psi_{q-1}(t))=3+q-2=q+1$. 
Hence, $\deg(f_{p,q}(t))=pq-2-(q+1)=(p-1)q-3$. 

Finally, it can be directly checked that $f_{p,q}(t)$ satisfies the relation $t^{(p-1)q-3}f_{p,q}(t^{-1})=f_{p,q}(t)$. 
This shows the palindromicity of $f_{p,q}(t)$. 
\end{proof}

\begin{ex}\label{ex:fpq}
For small $p,q$, let us compute $f_{p,q}(t)$ explicitly: 
\begin{align*}
f_{5,5}(t)&=1+3t+6t^2+10t^3+14t^4+18t^5+22t^6+25t^7+26t^8+26t^9+\cdots+3t^{16}+t^{17}, \\
f_{5,7}(t)&=1+3t+6t^2+10t^3+14t^4+18t^5+22t^6+26t^7+30t^8+34t^9+37t^{10}+39t^{11}+40t^{12} \\
&\quad +40t^{13}+\cdots+3t^{24}+t^{25}, \\
f_{7,7}(t)&=1+3t+6t^2+10t^3+15t^4+21t^5+27t^6+33t^7+39t^8+45t^9+51t^{10}+57t^{11}+62t^{12} \\
&\quad +67t^{13}+72t^{14}+77t^{15}+82t^{16}+86t^{17}+88t^{18}+89t^{19}+89t^{20}+\cdots+3t^{38}+t^{39}. 
\end{align*} 
Then we see that 
\begin{align*}
(1-t)f_{5,7}(t)&=1+2t+3t^2+4t^3+4(t^4+\cdots+t^9)+3t^{10}+2t^{11}+t^{12}-t^{14}-\cdots-t^{26}. 
\end{align*}
Namely, the first parts of the coefficients $c_0,c_1,\ldots,c_{12}$ of $(1-t)f_{5,7}(t)$ are 
\begin{align*}
c_i=\begin{cases}
i+1 &\text{if }i=0,1,\ldots,3, \\
4 &\text{if }i \in \{4,\ldots,9\}, \\
3 &\text{if }i \in \{10\}, \\
13 - i &\text{if }i=11,12. 
\end{cases} 
\end{align*}
This description matches \eqref{eq:c_i} for $(p,q)=(5,7)$ appearing in the proof of Proposition~\ref{prop:simp_comp}. 

Moreover, we have 
\begin{align*}
(1-t)f_{7,7}(t)&=1+2t+3t^2+4t^3+5t^4+6t^5+6(t^6+\cdots+t^{11})+5(t^{12}+\cdots+t^{16}) \\
&\quad+4t^{17}+2t^{18}+t^{19}-t^{21}-\cdots-t^{40}. 
\end{align*}
Namely, the first parts of the coefficients $c_0,c_1,\ldots,c_{19}$ of $(1-t)f_{7,7}(t)$ are 
\begin{align*}
c_i=\begin{cases}
i+1 &\text{if }i=0,1,\ldots,5, \\
6 &\text{if }i \in \{6,\ldots,11\}, \\
5 &\text{if }i \in \{12,\ldots,16\}, \\
4 &\text{if }i \in \{17\}, \\
20 - i &\text{if }i=18,19, 
\end{cases} 
\end{align*}
This description matches \eqref{eq:c_i'} for $p=q=7$ appearing in the proof of Proposition~\ref{prop:simp_comp}. 
\end{ex}

\begin{rem}
In fact, one can verify that the following equality holds:
\begin{align*}
f_{p,q}(t)=&\Psi_{p-1}(t)\left(\Psi_{(p-1)(q-1)/2}(t)\right)^2 - t^{p+q-2}\Psi_{(p-3)q/2-p+1}(t)\Psi_{(p-3)q/2-p+2}(t) \\ 
&- t^{2(q-1)+p-1}\Psi_{(p-5)(q-1)/2-(p-1)}(t)\Psi_{(p-5)(q-1)/2-1}(t) \\
&- \sum_{k=3}^{(p-3)/2} t^{k(q-1)+p-1}\Psi_{((p-1)/2 -k)(q-1)-(p+2-k)}(t)\Psi_{((p-1)/2-k)(q-1)-k+2}(t) \\
&- t^{(p-1)(q-1)/2-1}(1+t)\Psi_{(p+3)/2}(t)\Psi_{(p-5)/2}(t) 
\end{align*}
for $7 \leq p \leq q$, with the convention $\Psi_0(t)=0$, and 
\[f_{5,q}(t)=\Psi_4(t)\left(\Psi_{2(q-1)}(t)\right)^2 - t^{q+3}\Psi_{q-4}(t)\Psi_{q-3}(t) - t^{2q-3}\Psi_{4}(t) \]
for $q \geq 5$. We omit the proof of this identity, since it will not be used in what follows.
\end{rem}

The following constitutes the crucial part of the proof of Theorem~\ref{thm:p=3} (3), and hence of Theorem~\ref{thm:main} (2).
\begin{prop}\label{prop:simp_comp}
Let $5 \leq p \leq q$ be odd integers. 
Then there exists a simplicial polytope of dimension $(p-1)q-3$ with $(p-1)q$ vertices whose boundary complex $\Delta$ has the $h$-polynomial $f_{p,q}(t)$. 
\end{prop}
\begin{proof}
Let $f_{p,q}(t)=\sum_{i = 0}^sh_it^i$, where $s=(p-1)q-3$, and let $(1-t)f_{p,q}(t)=\sum_{i \geq 0}c_it^i$. Then $c_i=h_i-h_{i-1}$ for $i=1,\ldots,(p-1)q-3$. 
Note that $h_1=3$. This implies that the number of vertices $f_0(\Delta)$ of $\Delta$ (a desired simplicial complex) is $(p-1)q-3+h_1=(p-1)q$ (see \eqref{eq:fh}). 

By a direct computation of the coefficients of the polynomial $(1-t)f_{p,q}(t)$, which coincides with $\displaystyle\frac{g_{p,q}(t)}{t(1-t)\Psi_{q-1}(t)}$, 
we see that the first half of the coefficients of $(1-t)f_{p,q}(t)$, i.e., $c_0=1,c_1,\ldots,c_{(p-1)q/2-2}$ are of the following form: if $5 \leq p < q$, then 
\begin{align}\label{eq:c_i}
c_i=\begin{cases}
i+1 &\text{if }i=0,1,\ldots,p-2, \\
p-1-k &\text{if }i \in I_k \text{ for }k=0,1,\ldots,(p-3)/2, \\
(p-1)q/2-1 - i &\text{if }i=(p-1)(q-1)/2-1,\ldots,(p-1)q/2 -2, 
\end{cases} 
\end{align}
where 
\begin{align*}
I_k&=\{k(q-1)+p-1,\ldots,k(q-1)+p-1+(q-2)\} \text{ for }k=0,1,\ldots,(p-5)/2, \text{ and } \\
I_{(p-3)/2}&=\{(p-3)(q-1)/2 + p-1, \ldots,\underbrace{(p-3)(q-1)/2+p-1+q-p-2}_{=(p-1)(q-1)/2 -2}\}; 
\end{align*}
and if $5 \leq p=q$, then 
\begin{align}\label{eq:c_i'}
c_i=\begin{cases}
i+1 &\text{if }i=0,1,\ldots,p-2, \\
p-1-k &\text{if }i \in I_k \text{ for }k=0,1,\ldots,(p-3)/2, \\
(p-1)q/2-1 - i &\text{if }i=(p-1)(q-1)/2,\ldots,(p-1)q/2 -2, 
\end{cases} 
\end{align}
where 
\begin{align*}
I_k&=\{k(q-1)+p-1,\ldots,k(q-1)+p-1+q-2\} \text{ for }k=0,1,\ldots,(p-7)/2, \\
I_{(p-5)/2}&=\{(p-5)(q-1)/2 + p-1,\ldots,\underbrace{(p-5)(q-1)/2 +p-1+q-3}_{=(p-1)(q-1)/2-2}\}, \text{ and } \\
I_{(p-3)/2}&=\{(p-1)(q-1)/2 -1\}. 
\end{align*}
See Example~\ref{ex:fpq} for illustrations. 

Note that $f_{p,q}(t)$ is a palindromic polynomial whose constant term is $1$ (see Lemma~\ref{lem:fpq}).
Since $h_0=1$ and $h_1=3$, we have $h_1-h_0=2\geq 0$.
Hence, according to Theorem~\ref{thm:g}, in order for there to exist a simplicial polytope whose $h$-polynomial is $f_{p,q}(t)$, it suffices to verify the following inequalities:
\[0 \leq c_{i+1} \leq c_i^{\langle i \rangle} \;\text{ for each }\;i=1,\ldots,(p-1)q/2-2.\]

The first inequality $0 \leq c_{i+1}$ holds for each $i$. Regarding the second one, we can get a binomial sum expression as follows: 
\begin{align*}
c_i=\begin{cases}
\binom{i+1}{i} &\text{if }i=0,1,\ldots,p-2, \\
\binom{i}{i}+\binom{i-1}{i-1}+\cdots+\binom{i-c_i+1}{i-c_i+1} &\text{if }i=p-1,\ldots,(p-1)q/2-2, 
\end{cases}
\end{align*}
where we notice that $c_i \leq i$ if $i \geq p-1$. Hence, 
\begin{align*}
c_i^{\langle i \rangle}=\begin{cases}
\binom{i+2}{i+1}=i+2=c_{i+1} &\text{if }i=0,1,\ldots,p-3, \\
\binom{p}{p-1}=p=c_{p-1}+1 &\text{if }i=p-2, \\
\binom{i+1}{i+1}+\binom{i}{i}+\cdots+\binom{i-c_i+2}{i-c_i+2}=c_i &\text{if }i=p-1,\ldots,(p-1)q/2-2. 
\end{cases}
\end{align*}
Since $c_i \geq c_{i+1}$ holds for each $i$ with $p-2 \leq i \leq (p-1)q/2-2$ (see \eqref{eq:c_i} and \eqref{eq:c_i'}), we conclude the desired condition. 
\end{proof}

Now, we are ready to give a proof of Theorem~\ref{thm:p=3} (3).
\begin{proof}[Proof of Theorem~\ref{thm:p=3} (3)]
By the definition of $f_{p,q}(t)$, we have
\begin{align*}
\frac{\Psi_p(t)\Psi_q(t)\Phi_6(t^{(p-1)(q-1)/2})}{(1-t)^{pq-1}}
&=
\frac{\Psi_{pq}(t)}{(1-t)^{pq-1}}
+\frac{\Psi_{q-1}(t)f_{p,q}(t)}{(1-t)^{pq-3}} 
-\frac{\Psi_{q-1}(t)f_{p,q}(t)}{(1-t)^{pq-4}}.
\end{align*}
Since $h_{\partial \Delta_{pq}}(t)=\Psi_{pq}(t)$ and
$\partial \Delta_{pq}$ has dimension $pq-2$, by the same argument as before,
it is enough to show the existence of a subcomplex of $\partial \Delta_{pq}$ of dimension $pq-5$ whose $h$-polynomial is equal to
$\Psi_{q-1}(t)f_{p,q}(t)$.
We can construct such a subcomplex as $\partial \Delta_{q-1} * \Delta' * \Delta_1$ inside 
$\partial \Delta_{pq}$, where $\Delta'$ is a simplicial complex guaranteed 
by Proposition~\ref{prop:simp_comp}, because the number of vertices is equal to
\[\underbrace{q-1}_{=\partial \Delta_{q-1}}+\underbrace{(p-1)q}_{=\Delta'}+\underbrace{1}_{=\Delta_1}=pq,\]
which is equal to the number of vertices of $\partial \Delta_{pq}$. 
Moreover, the dimension is equal to $q-3+(p-1)q-4+0+2=pq-5$.
Thus, there exists a simplicial complex $\Delta$ such that
$h_{\kk[\Delta]}(t)
=
\Psi_p(t)\Psi_q(t)\Phi_6(t^{(p-1)(q-1)/2})$.

Finally, set $m=(p-1)(q-1)/2$. Then $\Phi_{6m}(t)$ divides
$\Psi_p(t)\Psi_q(t)\Phi_6(t^m)$,
whereas $\Phi_{3m}(t)$ does not divide it, since
$p\leq q<3m$.
Hence, by Lemma~\ref{l.nonCI},
$\Psi_p(t)\Psi_q(t)\Phi_6(t^m)$ is not of type CI.
This completes the proof.
\end{proof}

We can now prove Theorem~\ref{thm:main} (2). 
\begin{proof}[Proof of Theorem~\ref{thm:main} (2)] 
Let $n \geq 8$ be a composite integer. 
In what follows, we construct a simplicial complex whose $h$-polynomial $h(t)$ is not of type CI with $h(1)=n$. 
We divide the discussion into two cases: 
(A) $n$ has a prime divisor at least $5$, or (B) $n$ is of the form $n=2^i 3^j$.

(A) Let $n=\prod_{i=1}^b p_i$, where each $p_i$ is prime with $p_1 \geq \cdots \geq p_b$. 
Then $b \geq 2$ and $p_1 \geq 5$ by assumption. 
Let $\widetilde{\Delta}$ be a simplicial complex whose $h$-polynomial is $\Psi_{p_1}(t)\Psi_{p_2}(t)\Phi_6(t^{(p_1-1)(p_2-1)/2})$. 
The existence of such simplicial complex is guaranteed by the proof of Theorem~\ref{thm:p=3}. 
Let $\Delta=\widetilde{\Delta} * \partial \Delta_{p_3} * \cdots * \partial \Delta_{p_b}$. 
Then \[h_\Delta(t)=\left(\prod_{i=1}^b \Psi_{p_i}(t)\right) \cdot \Phi_6(t^m)\;\text{ and }\; h_\Delta(1)=n,\] 
where $m=(p_1-1)(p_2-1)/2$. Then $\Phi_{6m}(t)$ divides $h_\Delta(t)$. On the other hand, since $3m>p_1 \geq p_i$, $\Phi_{3m}(t)$ does not divide $h_\Delta(t)$. 
These imply that $h_\Delta(t)$ is not of type CI by Lemma~\ref{l.nonCI}. 

(B) Let $n=2^i3^j \geq 8$. Then either $8 \mid n$, $9 \mid n$ or $12 \mid n$. 
In the case $8 \mid n$, i.e., $i \geq 3$, let \[\Delta=\Delta^{(1)} * 
\underbrace{\partial \Delta_2 * \cdots * \partial \Delta_2}_{i-3} * \underbrace{\partial \Delta_3 * \cdots * \partial \Delta_3}_j,\]
where $\Delta^{(1)}$ is the simplicial complex given in Example~\ref{ex_8912}. 
Then $h_\Delta(t)=\Psi_2(t)^i \Psi_3(t)^j \Phi_{10}(t)$ and $h_\Delta(1)=2^i3^j=n$. 
This resulting $h$-polynomial is divisible by $\Phi_{10}(t)$ but not by $\Phi_5(t)$. Hence, this is not of type CI. 
In the cases $9 \mid n$ and $12 \mid n$, i.e., $j \geq 2$ and $i \geq 2,j\geq 1$, we may set $\Delta$ by 
\begin{align*}
\Delta&=\Delta^{(2)} * \underbrace{\partial \Delta_2 * \cdots * \partial \Delta_2}_i * \underbrace{\partial \Delta_3 * \cdots * \partial \Delta_3}_{j-2}, \;\text{ and}\\
\Delta&=\Delta^{(3)} * \underbrace{\partial \Delta_2 * \cdots * \partial \Delta_2}_{i-2} * \underbrace{\partial \Delta_3 * \cdots * \partial \Delta_3}_{j-1}, 
\end{align*} 
where $\Delta^{(2)}$ and $\Delta^{(3)}$ are the ones given in Example~\ref{ex_8912} (2) and (3), respectively.  
The resulting $h$-polynomial in the case $9 \mid n$ (resp. $12 \mid n$) is divisible by $\Phi_{12}(t)$ (resp. $\Phi_{18}(t)$) but not by $\Phi_4(t)$ (resp. $\Phi_6(t)$). 

Therefore, for any composite integer $n \geq 8$, 
we conclude the existence of a cyclotomic standard graded $\kk$-algebra $R$ whose $h$-polynomial $h_R(t)$ is not of type CI and satisfies $h_R(1)=n$. 
\end{proof}

What we have shown in this section is that, if $p=2$ and $q \geq 5$ is odd, or if $p,q \geq 3$ are both odd,  
then there exists a standard graded $\kk$-algebra, in fact a Stanley--Reisner ring, whose $h$-polynomial is equal to $\Psi_p(t)\Psi_q(t)\Phi_6(t^{(p-1)(q-1)/2})$.   
It is natural to consider a ``more direct'' application of Macaulay's theorem to verify the existence of a standard graded $\kk$-algebra having this $h$-polynomial.  
However, such an approach seems to be highly nontrivial. 
Indeed, one must translate the coefficients of the polynomial $\Psi_p(t)\Psi_q(t)\Phi_6(t^{(p-1)(q-1)/2})$ into the language of Hilbert functions via \eqref{eq:hilb}, 
and the resulting expressions quickly become intricate.  
For instance, let $p=2$ and $q \geq 5$ be odd. Then  
\begin{align*}
\Psi_2(t)\Psi_q(t)\Phi_6(t^{(q-1)/2})&=1+2(t+\cdots+t^{(q-3)/2})+t^{(q-1)/2}+t^{q-1}+t^q+t^{(3q-1)/2}\\
&\quad +2(t^{(3q+1)/2}+\cdots+t^{2q-2})+t^{2q-1}.
\end{align*}
Let $d$ be a sufficiently large integer, and define 
\[f(n)=\sum_{i \geq 0}h_i\binom{d-1+n-i}{n-i}\]
with \eqref{eq:hilb} in mind, 
where $h_i$ denotes the coefficient of $t^i$ in $\Psi_2(t)\Psi_q(t)\Phi_6(t^{(q-1)/2})$.
To apply Macaulay's theorem, we need to verify that
\begin{align}\label{eq.Mac}
f(n+1) \leq f(n)^{\langle n \rangle} \quad \text{for all $n \geq 1$.}
\end{align}
Although these inequalities are guaranteed by our result (the proof of Theorem~\ref{thm:p=3}), we now attempt a direct verification. For example, we have 
\begin{align*}
f(q-2)&=\binom{d-1+q-2}{q-2}+2\left(\binom{d-1+q-3}{q-3}+\cdots+\binom{d-1+(q-1)/2}{(q-1)/2}\right)\\
&\quad+\binom{d-1+(q-3)/2}{(q-3)/2},\\
f(q-1)&=\binom{d-1+q-1}{q-1} +2\left(\binom{d-1+q-2}{q-2}+\cdots+\binom{d-1+(q+1)/2}{(q+1)/2}\right)\\
&\quad+\binom{d-1+(q-1)/2}{(q-1)/2}+\binom{d-1}{0},
\end{align*}
and one needs to estimate $f(q-2)^{\langle q-2 \rangle}-f(q-1)$,
which requires explicit and complicated computations of the binomial sum expression of $f(q-2)^{\langle q-2 \rangle}$.
Consequently, a direct verification of \eqref{eq.Mac} seems quite difficult.
Our approach, which consists of finding a tractable polynomial and applying the $g$-theorem, is technical but significant, as it allows us to avoid elaborate computations of binomial sum expressions.


\begin{thebibliography}{99}
\bibitem{BD} A. Borz\`i and A. D'Al\`i,
Graded algebras with cyclotomic Hilbert series,
\textit{J. Pure Appl. Algebra} \textbf{225} (2021), no. 12, Paper No. 106764, 9 pp.

\bibitem{Br} P. Br\"{a}nd\'{e}n,
``Unimodality, log-concavity, real-rootedness and beyond'',
in: Handbook of Enumerative Combinatorics, Discrete Math. Appl. (Boca Raton), CRC Press, Boca Raton, FL, 2015, pp. 437--483.

\bibitem{BL} B. Braun and F. Liu, 
$h^*$-polynomials with roots on the unit circle, 
\textit{Exp. Math.} \textbf{30} (2021), no.~3, 332--348. 

\bibitem{BH} W. Bruns and J. Herzog, 
``Cohen-Macaulay rings, revised edition'', 
Cambridge Studies in Advanced Mathematics, 39, Cambridge University Press, Cambridge, 1998.

\bibitem{H} A. Higashitani,
Almost Gorenstein homogeneous rings and their $h$-vectors,
\textit{J. Algebra} \textbf{456} (2016), 190--206. 

\bibitem{JZ} P. J\o rgensen and J. J. Zhang,
Gourmet's guide to Gorensteinness,
\textit{Adv. Math.} \textbf{151} (2000), no. 2, 313--345.

\bibitem{KKZ} E. Kirkman, J. Kuzmanovich and J. J. Zhang,
Noncommutative complete intersections,
\textit{J. Algebra} \textbf{429} (2015), 253--286. 

\bibitem{KKZ2} E. Kirkman, J. Kuzmanovich and J. J. Zhang,
Invariant theory of finite group actions on down-up algebras,
\textit{Transform. Groups} \textbf{20} (2015), no. 1, 113--165.

\bibitem{L} S. Lang,
``Algebra, revised third edition'',
Grad. Texts in Math., 211, Springer-Verlag, New York, 2002.

\bibitem{M} S. Miyashita,
When do pseudo-Gorenstein rings become Gorenstein?,
\textit{Bull. Lond. Math. Soc.} \textbf{58} (2026), no. 1, Paper No. e70186, 14 pp.

\bibitem{S78} R. P. Stanley,
Hilbert functions of graded algebras,
\textit{Adv. Math.} \textbf{28} (1978), no. 1, 57--83.
 
\bibitem{Sbook} R. P. Stanley,
``Combinatorics and commutative algebra, second edition'',
Progr. Math., 41, Birkh\"auser Boston, Inc., Boston, MA, 1996. 
\end{thebibliography}
\end{document}